\documentclass[11pt,leqno]{article}
\usepackage{mdwlist}
\usepackage{amsmath}
\usepackage{amssymb}
\usepackage{theorem}
\newcommand{\bDK}{$\bD{}\,$-Klee}
\newcommand{\fDK}{$\fD{}\,$-Klee}
\newcommand{\nnn}{\ensuremath{{n\in{\mathbb N}}}}
\newcommand{\thalb}{\ensuremath{\tfrac{1}{2}}}
\newcommand{\menge}[2]{\big\{{#1}~\big |~{#2}\big\}}

\newcommand{\To}{\ensuremath{\rightrightarrows}}

\newcommand{\fenv}[1]%
{\ensuremath{\,\overrightarrow{\operatorname{env}}_{#1}}}
\newcommand{\benv}[1]%
{\ensuremath{\,\overleftarrow{\operatorname{env}}_{#1}}}
\newcommand{\emp}{\ensuremath{\varnothing}}

\newcommand{\scal}[2]{\left\langle{#1},{#2}  \right\rangle}

\newcommand{\RR}{\ensuremath{\mathbb R}}

\newcommand{\RPX}{\ensuremath{\left[0,+\infty\right]}}

\newcommand{\RX}{\ensuremath{\,\left]-\infty,+\infty\right]}}

\newcommand{\oldIDD}{\ensuremath{\operatorname{int}\operatorname{dom}f}}
\newcommand{\IDD}{\ensuremath{U}}

\newcommand{\dom}{\ensuremath{\operatorname{dom}}}

\newcommand{\argmax}{\ensuremath{\operatorname*{argmax}}}

\newcommand{\intdom}{\ensuremath{\operatorname{int}\operatorname{dom}}\,}

\newcommand{\inte}{\ensuremath{\operatorname{int}}}
\newcommand{\deriv}{\ensuremath{\operatorname{\; d}}}
\newcommand{\rockderiv}{\ensuremath{\operatorname{\; \hat{d}}}}
\newcommand{\closu}{\ensuremath{\operatorname{cl}}}

\newcommand{\ran}{\ensuremath{\operatorname{ran}}}
\newcommand{\conv}{\ensuremath{\operatorname{conv}}}
\newcommand{\cconv}{\ensuremath{\overline{\operatorname{conv}}}}

\newcommand{\Id}{\ensuremath{\operatorname{Id}}}

\newcommand{\Hess}{\ensuremath{\nabla^2\!}}

\newcommand{\ffproj}[1]{\overrightarrow{Q\thinspace}_%
{\negthinspace\negthinspace #1}}

\newcommand{\bfproj}[1]{\overleftarrow{\thinspace Q\thinspace}_%
{\negthinspace\negthinspace #1}}

\newcommand{\fD}[1]{\overrightarrow{D\thinspace}_%
{\negthinspace\negthinspace #1}}
\newcommand{\ffD}[1]{\overrightarrow{F\thinspace}_%
{\negthinspace\negthinspace #1}}
\newcommand{\bD}[1]{\overleftarrow{\thinspace D\thinspace}_%
{\negthinspace\negthinspace #1}}
\newcommand{\bfD}[1]{\overleftarrow{\thinspace F\thinspace}_%
{\negthinspace\negthinspace #1}}

\newcommand{\pinf}{\ensuremath{+\infty}}

{\begin{list}{}{%
\settowidth{\labelwidth}{\textrm{#1~}}%
\setlength{\leftmargin}{\labelwidth+\labelsep}}}
{\end{list}}
\newtheorem{theorem}{Theorem}[section]
\newtheorem{lemma}[theorem]{Lemma}
\newtheorem{corollary}[theorem]{Corollary}
\newtheorem{proposition}[theorem]{Proposition}
\newtheorem{definition}[theorem]{Definition}
\theoremstyle{plain}{\theorembodyfont{\rmfamily}
}
\theoremstyle{plain}{\theorembodyfont{\rmfamily}
}
\theoremstyle{plain}{\theorembodyfont{\rmfamily}
}
\theoremstyle{plain}{\theorembodyfont{\rmfamily}
\newtheorem{example}[theorem]{Example}}
\newtheorem{fact}[theorem]{Fact}
\theoremstyle{plain}{\theorembodyfont{\rmfamily}
\newtheorem{remark}[theorem]{Remark}}

\begin{document}

\title{\sffamily Bregman distances and Klee sets}

\author{Heinz H.\ Bauschke\thanks{Mathematics, Irving K.\ Barber School,
The University of British Columbia Okanagan, Kelowna,
B.C. V1V 1V7, Canada. Email:
\texttt{heinz.bauschke@ubc.ca}.},~
Xianfu
Wang\thanks{Mathematics, Irving K.\ Barber School,
The University of British Columbia Okanagan, Kelowna,
B.C. V1V 1V7, Canada.
Email:
\texttt{shawn.wang@ubc.ca}.},~
Jane Ye\thanks{Department of Mathematics and Statistics, University of
Victoria, Victoria, B.C. V8P 5C2, Canada.
Email:~\texttt{janeye@math.uvic.ca}.},~
and~Xiaoming Yuan\thanks{Department of Management Science,
Antai College of Economics and Management,
Shanghai Jiao Tong University, Shanghai, 200052, China.
Email: \texttt{xmyuan@sjtu.edu.cn}.}}

\date{February 16, 2008}

\maketitle

\vskip 8mm

\begin{abstract} \noindent
In 1960, Klee showed that
a subset of a Euclidean space
must be a singleton provided that each point in the space has a unique
farthest point in the set.
This classical result has received much attention; in fact, the Hilbert
space version is a famous open problem.
In this paper, we consider Klee sets from a new perspective. Rather than
measuring distance induced by a norm, we focus on the case when distance
is meant in the sense of Bregman, i.e., induced by a convex function.
When the convex function has sufficiently nice properties, then ---
analogously to the Euclidean distance case --- every Klee set must be a
singleton. We provide two proofs of this result, based on Monotone Operator
Theory and on Nonsmooth Analysis.
The latter approach leads to results that complement
work by Hiriart-Urruty on the Euclidean case.
\end{abstract}

{\small

 }

\section{Introduction}
Throughout this paper, $\RR^J$ denotes
the standard Euclidean space with inner
product $\scal{\cdot}{\cdot}$ and induced norm $\|\cdot\|$.
Let $C$ be a nonempty bounded closed subset of $\RR^J$ and assume
that $C$ is a \emph{Klee set (with respect to the Euclidean distance)}, i.e.,
each point in $\RR^J$ has a unique farthest point in $C$.
Must $C$ be a singleton?
The \emph{farthest-point conjecture} \cite{urruty2}
proclaims an affirmative answer to this question.
This conjecture has attracted many mathematicians; see, e.g.,
\cite{lewis,deutsch,urruty2,urruty3,urruty1,west}
and the references therein.
Although the farthest-point conjecture is true in $\RR^J$,
as was shown originally by Klee \cite{klee}
(see also \cite{Edgar1,urruty2,MSV}),
only partial results are known
in infinite-dimensional settings
(see, e.g., \cite{panda,west}).

In this paper, we cast a new light on this problem by measuring
the distance in the sense of Bregman rather than in the usual Euclidean
sense. To this end, assume that
\begin{equation} \label{eq:preD}
\text{$f \colon \RR^J\to\RX$ is convex and differentiable on $U :=
\oldIDD \neq \emp$,}
\end{equation}
where $\intdom f$ stands for the interior of the set
$\dom f:= \menge{x\in\RR^J}{f(x)\in\RR}$.
Then the \emph{Bregman distance} \cite{Bregman} with respect to $f$, written $D_f$ or
simply $D$, is
\begin{equation} \label{eq:D}
D\colon \RR^J \times \RR^J \to \RPX \colon (x,y) \mapsto
\begin{cases}
f(x)-f(y)-\scal{\nabla f(y)}{x-y}, &\text{if}\;\;y\in\IDD;\\
\pinf, & \text{otherwise}.
\end{cases}
\end{equation}
Although standard, it is well known
that the name ``Bregman distance'' is somewhat misleading
because in general $D$ is neither symmetric nor does the triangle
inequality hold.
We recommend the books \cite{ButIus,CenZen} to the reader for
further information on
Bregman distances and their various applications.

Throughout, we assume that
\begin{equation}
C\subset U.
\end{equation}
Now define
the \emph{left Bregman farthest-distance function} by
\begin{equation} \label{e:montag:a}
\bfD{C}\: \colon U\to \RPX \colon  y\mapsto \sup_{x\in C}D(x,y),
\end{equation}
and the corresponding \emph{left Bregman farthest-point map} by
\begin{equation} \label{e:montag:b}
\bfproj{C} \colon \IDD \to \IDD \colon
y\mapsto \underset{x\in C}{\operatorname{argmax}}\;\:D(x,y).
\end{equation}
Since $D$ is in general not symmetric, there exist analogously
the \emph{right Bregman farthest-distance function} and
the \emph{right Bregman farthest-point map}.  These objects,
which we will study later,
are denoted by $\ffD{C}$ and $\ffproj{C}$, respectively.
When $f = \thalb\|\cdot\|^2$, then $D\colon (x,y)\mapsto
\thalb\|x-y\|^2$ is symmetric and the corresponding map $\bfproj{C}$
is identical to the farthest-point map with respect to the Euclidean
distance.

The present more general framework based on Bregman distances allows
for significant extensions of Hiriart-Urruty's work \cite{urruty2}
(and for variants of some of the results in \cite{west}).
One of our main result states that if $f$ is sufficiently nice,
then every Klee set (with respect to $D$) must be a singleton.
Two fairly distinct proofs of this result are given.
The first is based on the deep Br\'ezis-Haraux range approximation theorem
from monotone operator theory.
The second proof, which  uses
generalized subdifferentials from nonsmooth analysis,
allows us to characterize sets with unique farthest points.
Various subdifferentiability properties  of the Bregman farthest-distance
function are also provided.
The present work complements a corresponding study on Chebyshev sets
\cite{BWYY1}, where
the focus is on \emph{nearest} rather than farthest points.

The paper is organized as follows.
In Section~\ref{assumption}, we state our standing assumptions and we
provide some concrete examples for $f$.
In Section~\ref{Lipschitz}, Bregman farthest points are characterized
and it is shown that the Bregman farthest-distance function
is locally Lipschitz.
The first proof of our main result is presented in Section~\ref{monotone}.
In Section~\ref{subdiff}, we study subdifferentiabilities of
farthest-distance function.
We establish Clarke regularity, and we provide
an explicit formula for the Clarke subdifferential.
Section~\ref{completechris} contains several  characterizations
of Klee sets. The results extend
Hiriart-Urruty's work \cite{urruty2} from Euclidean to Bregman distances.
In the final  Section~\ref{finalright}, we show that
the right Bregman farthest-point map $\ffproj{C}^f$ can be studied in terms
of the left and dual counterpart $\bfproj{\nabla f(C)}^{f^*}$.
When $f$ is sufficiently nice, this allows us to deduce that
Klee sets with respect to the right Bregman farthest-point map are
necessarily singletons.

We employ standard notation from Convex Analysis;
see, e.g., \cite{Rock70,Rock98,Zali}.
For a function $h$, the subdifferential in the convex-analytical
sense is denoted by $\partial h$, $h^*$ stands for the Fenchel
conjugate, and $\dom h$ is the set of all points where $h$ is not $\pinf$.
If $h$ is differentiable at $x$, then
$\nabla h(x)$ and $\nabla^2 h(x)$ denotes the gradient vector
and the Hessian matrix at $x$, respectively.
The notation $\conv h$ ($\cconv h$) denotes
the convex hull (closed convex hull) of $h$.
For a set $S$, the
expressions $\inte S$, $\closu S$, $\conv S$, $\cconv S$
signify the interior,
the closure, the convex hull, and the closed convex hull of $S$,
respectively.
A set-valued operator $T$ from $X$ to $Y$, is written as
$T\colon X\To Y$, and $\dom T$ and $\ran T$ stand for the
domain and range of $T$.
Finally, we simply write $\varliminf$ and $\varlimsup$,
for the limit inferior and limit superior (as they occur in
and in set-valued analysis).

\section{Standing Assumptions and Examples}\label{assumption}

From now on,
and until the end of Section~\ref{completechris},
our standing assumptions are:
\begin{itemize}
\item[\bfseries A1]
The function $f\colon\RR^J\to\RX$ is a \emph{convex function of Legendre type},
i.e., $f$ is essentially smooth and essentially strictly convex in the sense
of Rockafellar \cite[Section~26]{Rock70}, with $U := \intdom f$.
\item[\bfseries A2]
The function $f$ is \emph{$1$-coercive} (also known as
\emph{supercoercive}), i.e., $\textstyle \lim_{\|x\|\rightarrow
+\infty}f(x)/\|x\|=+\infty$. An equivalent requirement is
$\dom f^*=\RR^J$ (see \cite[Theorem~11.8(d)]{Rock98}).
\item[\bfseries A3] The set $C$ is a nonempty bounded closed
(hence compact) subset of $U$.
\end{itemize}

There are many instances of functions satisfying \textbf{A1}--\textbf{A3}.
We list only a few.

\begin{example} \label{ex:dienstag}
Let $x=(x_j)_{1\leq j\leq J}$ and $y=(y_j)_{1\leq j\leq J}$ be two
points in $\RR^J$.
\begin{enumerate}
\item \label{ex:dienstag:energy}
\emph{Energy:}
If $f=\thalb\|\cdot\|^2$, then $U=\RR^J$, $f^*=f$, and
\begin{equation*}
D(x,y) = \thalb\|x-y\|^2.
\end{equation*}
Thus, the Bregman distance with respect to the energy corresponds to the
usual Euclidean distance.
\item  \label{ex:examples:KL}
\emph{Boltzmann-Shannon Entropy:} If $f\colon
x\mapsto\sum_{j=1}^{J}x_j\ln(x_j)-x_j$ if $x\geq 0$, $+\infty$
otherwise (where $x\geq 0$ and $x>0$ is understood coordinate-wise and
$0\ln0 := 0$), then
$U=\menge{x\in \RR^J}{x>0}$,
$f^*(y)=\sum_{j=1}^J\exp(y_j)$, and
\begin{equation*}
D(x,y) = \begin{cases}
\textstyle \sum_{j=1}^J x_j \ln(x_j/y_j) -
x_j + y_j, &
\text{if $x \geq 0$ and $y>0$;}\\
\pinf, & \text{otherwise}
\end{cases}
\end{equation*}
is the famous \emph{Kullback-Leibler Divergence}.
\item More generally, given a function $\phi\colon\RR\to\RX$
satisfying \textbf{A1}--\textbf{A2} and setting
$f(x)=\sum_{j=1}^{J}\phi(x_j)$, one has the same properties for $f$,
with $U=(\inte\dom\phi)^{J}$ and
\begin{equation*}
D(x,y)=\sum_{j=1}^{J}\phi(x_j)-\phi(y_{j})-\phi'(y_j)(x_j-y_j).
\end{equation*}
For instance, one may consider $\phi\colon t\mapsto |t|^p/p$, where $p>1$.
\end{enumerate}
\end{example}

The following result recalls a key property of Legendre functions.

\begin{fact}[Rockafellar] \emph{\cite[Theorem~26.5]{Rock70}} \label{isom}
If $h$ is a convex function of Legendre type, then so is $h^*$ and
\begin{equation*}
\nabla h:\intdom h \to \intdom h^*
\end{equation*}
is a topological isomorphism with inverse mapping
$(\nabla h)^{-1}=\nabla h^*$.
\end{fact}
\begin{corollary} \label{c:isom}
The mappings $\nabla f\colon U\to\RR^J$ and
$\nabla f^*\colon \RR^J\to U$ are continuous, bijective, and
inverses of each other.
\end{corollary}

\section{Left Bregman Farthest Distances and Farthest Points}\label{Lipschitz}

The following result generalizes Hiriart-Urruty's
\cite[Proposition~3.1 and Corollary~3.2]{urruty2} and
provides a characterization of left Bregman farthest points
(recall \eqref{e:montag:a} and \eqref{e:montag:b}).

\begin{proposition}
Let $y\in U$, $x\in C$, and $\lambda\geq 1$. Then
\begin{equation}\label{farcha}
x\in \bfproj{C}(y)\qquad \Leftrightarrow \qquad
(\forall c\in C)\;\; D(c,x)\leq \langle\nabla f(y)-\nabla f(x),c-x\rangle.
\end{equation}
If $x\in\bfproj{C}(y)$ and
\begin{equation}\label{greatercase}
z_{\lambda}:=
\nabla f^*(\lambda \nabla f(y)+(1-\lambda)\nabla f(x)),
\end{equation}
then
$x\in\bfproj{C}(z_{\lambda})$; moreover,
if $\lambda >1$, then $\bfproj{C}(z_{\lambda})=\{x\}$.
\end{proposition}
\begin{proof}
By definition, $x\in\bfproj{C}(y)$ means that for each $c\in C$,
$0\geq D(c,y)-D(x,y)$, i.e.,
\begin{align}0 & \geq f(c)-f(x)-\langle\nabla f(y),c-x\rangle\notag\\
& =f(c)-f(x)-\langle\nabla f(x),c-x\rangle +\langle\nabla f(x)-\nabla
f(y),c-x\rangle\notag\\
& =D(c,x)-\langle \nabla f(y)-\nabla f(x),c-x\rangle.\notag
\end{align}
Hence \eqref{farcha} follows.
Now assume that $x\in\bfproj{C}(y)$ and take an arbitrary $c\in C$.
By \eqref{farcha},
\begin{equation}\label{positive}
\langle\nabla f(y)-\nabla f(x),c-x\rangle\geq 0.
\end{equation}
The definition of $z_\lambda$ and \eqref{positive} result in
\begin{equation}\label{bigger}
\langle\nabla f(z_{\lambda})-\nabla f(x),c-x\rangle
=\lambda\langle \nabla f(y)-\nabla f(x),c-x\rangle
\geq \langle \nabla f(y)-\nabla f(x),c-x\rangle.
\end{equation}
Now \eqref{farcha} and \eqref{bigger} imply
\begin{align}
D(c,x) &\leq\langle \nabla f(y)-\nabla f(x),c-x\rangle
\leq
\langle\nabla f(z_{\lambda})-\nabla f(x),c-x\rangle,
\end{align}
which --- again by
\eqref{farcha} --- yields that $x\in \bfproj{C}(z_{\lambda})$.
Finally, assume that $\lambda >1$ and let
$\hat{x}\in \bfproj{C}(z_{\lambda})$.
By \eqref{greatercase},
$x\in\bfproj{C}(z_{\lambda})$. Since
$D(x,z_{\lambda})=D(\hat{x},z_{\lambda})$,
we have
\begin{align*}
0 & =D(x,z_{\lambda})-D(\hat{x},z_{\lambda})\\
& =f(x)-f(\hat{x})-\langle\nabla f(z_{\lambda}),x-\hat{x}\rangle\\
& =f(x)-f(\hat{x})-\langle\lambda \nabla f(y)+(1-\lambda)\nabla f(x),x-\hat{x}\rangle
\end{align*}
so that
$$(1-\lambda)[f(x)-f(\hat{x})-\langle \nabla f(x),x-\hat{x}\rangle]
+\lambda [f(x)-f(\hat{x})-\langle\nabla f(y),x-\hat{x}\rangle]=0.$$
Then
$
(1-\lambda)[f(\hat{x})-f(x)-\langle\nabla f(x),\hat{x}-x\rangle]
=\lambda [f(x)-f(\hat{x})-\langle \nabla f(y),x-\hat{x}\rangle]$, and thus
$$(1-\lambda)D(\hat{x},x)=\lambda [D(x,y)-D(\hat{x},y)].$$
It follows that
\begin{equation} \label{e:montag:c}
D(x,y)-D(\hat{x},y)=\frac{1-\lambda}{\lambda}D(\hat{x},x).
\end{equation}
Assume that $x\neq \hat{x}$.
Then $D(\hat{x},x)>0$, and, since  $\lambda >1$, we get
$0>(1-\lambda) D(\hat{x},x)$.
In view of \eqref{e:montag:c}, we conclude
$D(x,y)<D(\hat{x},y)$,
which contradicts that $x$ is a farthest point of $y$.
Therefore, $x=\hat{x}$.
\end{proof}

It will be convenient to define $f^\vee = f\circ (-\Id)$, i.e.,
$f^{\vee}(y)=f(-y)$ for every $y\in \RR^J$.
Our standing assumptions \textbf{A1}--\textbf{A3} imply that the function
\begin{equation}\label{starting}
-f^{\vee}+\iota_{-C}\colon  \RR^J\rightarrow\RX\colon x \mapsto
\begin{cases}
-f(-x),&  \text {if $x\in -C$};\\
\pinf, & \text{otherwise}
\end{cases}
\end{equation}
is lower semicontinuous.
This function plays a role in our next result, where we show that
$\bfD{C}$ is a locally Lipschitz function on $U$.

\begin{proposition}\label{distance}
The left Bregman farthest distance function $\bfD{C}$ is continuous on $U$
and it can be written as the composition
\begin{equation} \label{e:montag:d}
\bfD{C} =\big(f^*+(-f^{\vee}+\iota_{-C})^*\big)\circ \nabla f, 
\end{equation}
where $f^*+(-f^{\vee}+\iota_{-C})^*$ is locally Lipschitz and
$\nabla f$ is continuous. 
Consequently, $\bfD{C}$ is locally Lipschitz on $U$ provided that
$\nabla f$ has the same property --- as  is the case when
$f$ is twice continuously differentiable.
Finally,
\begin{equation}\label{newconjugate}
(-f^{\vee}+\iota_{-C})^*=\bfD{C}\circ\nabla f^*-f^*,
\end{equation}
and hence
$\bfD{C}\circ \nabla f^*$ is a locally Lipschitz
convex function with full domain.
\end{proposition}
\begin{proof}
Fix $y\in U$. Then
\begin{align*}
\bfD{C}(y) &=\sup_{c\in C}[f(c)-f(y)-\langle\nabla f(y),c-y\rangle]\\
&= \sup_{c\in C}[f(c)-\langle \nabla f(y),c\rangle]+ f^*(\nabla f(y))\\
&=f^*(\nabla f(y))+\sup_{c\in C}[\langle \nabla f(y),-c\rangle -(-f)(c)]\\
&=f^*(\nabla f(y))+\sup_{c}[\langle \nabla f(y),-c\rangle -(-f(c)+\iota_{C}(c))]\\
&=f^*(\nabla f(y))+\sup_{z}[\langle \nabla f(y),z\rangle -(-f(-z)+\iota_{-C}(z))]\\
&=f^*(\nabla f(y))+ (-f^{\vee}+\iota_{-C})^*(\nabla f(y)).
\end{align*}
The assumptions \textbf{A1}--\textbf{A3}
imply that $-f^{\vee}+\iota_{-C}$ is
proper and $1$-coercive.
By \cite[Proposition~X.1.3.8]{urruty1},
the convex function
$(-f^{\vee}+\iota_{C})^*$ has full domain and
it thus is locally Lipschitz on $\RR^J$.
Since $f^*$ is likewise locally Lipschitz on $\RR^J$,
Fact~\ref{isom} yields the continuity of $\bfD{C}$.
The ``Consequently'' statement is a consequence of the Mean Value Theorem.
Finally, pre-composing \eqref{e:montag:d} by $\nabla f^*$ followed
by re-arranging yields \eqref{newconjugate},
which in turn shows that
$\bfD{C}\circ\nabla f^*$ is a locally Lipschitz convex function,
as it is the sum of two such functions.
\end{proof}

\section{Left Bregman Farthest-Point Maps }\label{monotone}

The next result contains some useful properties of the farthest point map
and item~(iii) is an extension of \cite[Proposition~3.3]{urruty2}.

\begin{proposition}\label{farthest}
Let $x$ and $y$ be in $U$. Then the following hold.
\begin{enumerate}
\item $\textstyle \bfproj{C}(x)\neq\emp$.
\item If $(x_n)_\nnn$ is a sequence in $U$ converging to $x$ and
$(c_n)_\nnn$ is a sequence in $C$ such that $(\forall \nnn)$
$c_{n}\in \bfproj{C}(x_{n})$, then
all cluster points of $(c_n)_\nnn$ lie in $\bfproj{C}(x)$.
Consequently,
$\bfproj{C}\colon U\To C$ is compact-valued and upper
semicontinuous (in the sense of set-valued analysis).
\item \label{farthest:e}
$\langle -\bfproj{C}(x)+\bfproj{C}(y), \nabla f(x)-\nabla
f(y)\rangle\geq 0$
and hence $-\bfproj{C}\circ\nabla f^*$ is monotone.
\end{enumerate}
\end{proposition}
\begin{proof}
(i): Since $D(\cdot,x)$ is continuous on $U$ and $C$ is compact subset of
$U$, it follows
that $D(\cdot,x)$ attains its supremum over $C$.

(ii): Suppose that $(x_n)_\nnn$ lies in $U$ and converges to $x$,
that $(c_n)_\nnn$ lies in $C$, and that $(\forall\nnn)$
$c_{n}\in \bfproj{C}(x_{n})$, i.e.,
\begin{equation} \label{e:montag:e}
(\forall\nnn)\quad
f(c_{n})-f(x_{n})-\langle \nabla
f(x_{n}),c_{n}-x_{n}\rangle= D(c_n,x_n)=\bfD{C}(x_{n}).
\end{equation}
By \textbf{A3}, $(c_n)_\nnn$ has cluster
points and they all lie in $C$. After passing to a subsequence if
necessary, we assume that $c_n\to \bar{c}\in C$.
Since $\bfD{C}$ is continuous on $U$ by Proposition~\ref{distance},
we pass to the limit in \eqref{e:montag:e} and deduce that
$f(\bar{c})-f(x)-\langle \nabla f(x),\bar{c}-x\rangle=D(\bar{c},x)=\bfD{C}(x)$.
Hence $\bar{c}\in \bfproj{C}(x)$.
The same reasoning (with $(x_n)_\nnn = (x)_\nnn$)
shows that $\bfproj{C}(x)$ is closed and hence compact (since $C$ is
compact).
Therefore, $\bfproj{C}$ is compact-valued and upper semicontinuous
 on $U$.

(iii): Let $p\in \bfproj{C}x$ and $q\in \bfproj{C}y$. Then
$D(p,x)\geq D(q,x)$ and $D(q,y)\geq D(p,y)$.
Using \eqref{eq:D}, we obtain
$f(p)-f(q)-\langle\nabla f(x), p-q\rangle \geq 0$ and
$f(q)-f(p)-\langle\nabla f(y), q-p\rangle \geq 0$.
Adding these two inequalities yields
$\langle\nabla f(x)-\nabla f(y), q-p\rangle\geq 0$.
The result now follows from Corollary~\ref{c:isom}.
\end{proof}

\begin{definition}
The set $C$ is \emph{Klee with respect to the left Bregman distance},
or simply \emph{\bDK},
if for every $x\in U$,
$\bfproj{C}(x)$ is nonempty and a singleton.
\end{definition}

\begin{proposition} \label{p:montag:f}
Suppose that $C$ is \bDK. Then $\bfproj{C}\colon U\to C$ is continuous.
Hence $-\bfproj{C}\circ \nabla f^*$ is continuous and maximal monotone.
\end{proposition}
\begin{proof}
By Proposition~\ref{farthest}(ii),
$\bfproj{C}$ is continuous on $U$.
This and the continuity of
$\nabla f^*:\RR^J\rightarrow U$ (see Corollary~\ref{c:isom}) imply that
$-\bfproj{C}\circ\nabla f^*$ is continuous.
On the other hand,
Proposition~\ref{farthest}(iii) shows that
$-\bfproj{C}\circ\nabla f^*$ is monotone.
Altogether, using \cite[Example~12.7]{Rock98},
we conclude that
$-\bfproj{C}\circ\nabla f^*$ is maximal monotone on $\RR^J$.
\end{proof}

The Br\'ezis-Haraux range approximation theorem plays a crucial role in the
proof of the following main result.
It is interesting to note that the Hilbert space analogue
\cite[Proposition~6.2]{west} by Westphal and Schwartz
relies only on the less powerful Minty's theorem.

\begin{theorem}[\bDK-sets are singletons]
Suppose that $C$ is \bDK. Then $C$ is a singleton.
\end{theorem}
\begin{proof}
Recall that Corollary~\ref{c:isom} and consider the following
two maximal monotone operators (see Proposition~\ref{p:montag:f})
$\nabla f^*$ and $-\bfproj{C}\circ\nabla f^*$.
The Br\'ezis-Haraux range approximation theorem
(see \cite[Section~19]{Simons}) implies that
\begin{equation} \label{e:montag:g}
\inte\ran\big(\nabla f^*-(\bfproj{C}\circ\nabla f^*)\big)=
\inte \big(\ran\nabla f^*-\ran(\bfproj{C}\circ\nabla f^*)\big)
=\inte\big( U -\ran(\bfproj{C}\circ\nabla f^*)\big).
\end{equation}
Since
$\ran(\bfproj{C}\circ\nabla f^*)\subseteq C$ and $C\subset U$,
we have
$0\in \inte (U-\ran(\bfproj{C}\circ\nabla f^*))$, and hence,
by \eqref{e:montag:g},
$0\in \inte \ran (\nabla f^*-(\bfproj{C}\circ\nabla f^*))$.
Thus there exists $x\in\RR^J$ such that
$\bfproj{C}(\nabla f^*(x))=\nabla f^*(x)$.
Hence $C$ must be a singleton.
\end{proof}

\begin{corollary} The set
$C$ is \bDK\ if and only if it is a singleton.
\end{corollary}

\section{ Subdifferentiability Properties}

\label{subdiff}

For a function $g$ that is finite and locally Lipschitz at a point
$y\in\RR^J$,
we define
the \emph{Dini subderivative} and \emph{Clarke subderivative} of $g$ at $y$
in the direction $w\in\RR^J$, denoted respectively
by $\deriv g(y)(w)$ and $\rockderiv g (y)(w)$, via
$$\deriv g(y)(w):=\varliminf_{t\downarrow 0}\frac{g(y+tw)-g(y)}{t},$$
$$\rockderiv g(y)(w):=\varlimsup_{\stackrel{x\rightarrow y}{t\downarrow
0}}\frac{g(x+tw)-g(x)}{t},$$
and the corresponding \emph{Dini subdifferential} and
\emph{Clarke subdifferential} via
$$\hat{\partial} g(y) :=\menge{y^*\in\RR^J}{(\forall w\in\RR^J)\;\;\langle y^*,w\rangle\leq \deriv g(y)(w)},$$
$$\overline{\partial} g(y) :=\menge{y^*\in\RR^J}{(\forall w\in\RR^J)\;\; \langle y^*,w\rangle\leq \rockderiv g(y)(w)}.$$
The \emph{limiting subdifferential} (see \cite[Definition~8.3]{Rock98})
is defined by
$$\partial_L g(y):=\varlimsup_{x\rightarrow y}\hat{\partial}
g(x).$$
We say that $g$ is \emph{Clarke
regular} at $y$ if $\deriv g(y)(w)=\rockderiv g(y)(w)$ for every
$w\in \RR^J$, or equivalently $\hat{\partial}
g(y)=\overline{\partial} g(y)$.
For further properties
of these subdifferentials and subderivatives,
see \cite{Frank,mordukhovich,Rock98}.

We now provide various subdifferentiability properties of $\bfD{C}$ in
terms of $\bfproj{C}$, and show that $\bfD{C}$ is {Clarke regular}.

\begin{proposition}[Clarke regularity]\label{howtofind}
Suppose that $f$ is twice continuously differentiable on $U$,
and let $y\in U$.
Then
\begin{equation}\label{regularderivative}
(\forall w\in\RR^J)\quad
\deriv \bfD{C}(y)(w)=\rockderiv
\bfD{C}(y)(w)=\max\langle \Hess
f(y)(y-\bfproj{C}(y)),w\rangle
\end{equation}
and
\begin{equation}\label{regularcase}
\partial_L \bfD{C}(y)=\hat{\partial}\bfD{C}(y)=\overline{\partial}
\bfD{C}(y)=\Hess f(y)[y-\conv \bfproj{C}(y)];
\end{equation}
consequently, $\bfD{C}$ is Clarke regular on $U$.
\end{proposition}

\begin{proof}
Set $g := \bfD{C}$ and let $x\in \bfproj{C}(y)$. Fix $w\in\RR^J$ and choose $t>0$ sufficiently small so that $y+tw\in U$. Since $x\in \bfproj{C}(y)$,
we note that
\begin{align*}
g(y+tw) &\geq f(x)-f(y+tw)-\langle\nabla f(y+tw),x-(y+tw)\rangle\\
& = f(x)-f(y+tw)-\langle\nabla f(y+tw),x-y\rangle +\langle\nabla f(y+tw),tw\rangle
\end{align*}
and
$g(y) = f(x)-f(y)-\langle\nabla f(y),x-y\rangle$.
Thus
$$\frac{g(y+tw)-g(y)}{t}\geq -\frac{f(y+tw)-f(y)}{t}-\frac{\langle \nabla f(y+tw)-\nabla f(y),x-y\rangle}{t}
+\langle \nabla f(y+tw), w\rangle.$$
Taking $\varliminf_{t\downarrow 0}$, we obtain
$\deriv g(y)(w)\geq -\langle \Hess f(y)w, x-y\rangle=\langle \Hess
f(y)(y-x),w\rangle$ and this implies
\begin{equation}\label{hamilton1}
\deriv g(y)(w)\geq \max \langle \Hess f(y)(y-\bfproj{C}(y)),w\rangle.
\end{equation}
Now take $x_{t}\in \bfproj{C}(y+tw)$ and estimate
$g(y+tw)= f(x_{t})-f(y+tw)-\langle\nabla f(y+tw),x_{t}-(y+tw)\rangle$
and $g(y)\geq f(x_{t})-f(y)-\langle\nabla f(y),x_{t}-y\rangle$.
Thus
\begin{equation}\label{bcday}
\frac{g(y+tw)-g(y)}{t}\leq -\frac{f(y+tw)-f(y)}{t}-\frac{\langle\nabla f(y+tw)-\nabla f(y),x_{t}-y\rangle}{t}
+\langle\nabla f(y+tw),w\rangle.
\end{equation}
Proposition~\ref{farthest}(ii) implies that as  $t\downarrow 0$,
all cluster points of $(x_{t})_{t>0}$ lie in
$\bfproj{C}(y)$.
Take a positive sequence
$(t_{n})_\nnn$ such that $t_n\downarrow 0$ and
$$\deriv g(y)(w)=\lim_{n\rightarrow\infty}\frac{g(y+t_{n}w)-g(y)}{t_{n}}.$$
After taking a subsequence if necessary,
we also assume that
$x_{t_{n}} \to x\in\bfproj{C}(y)$.
Then \eqref{bcday} implies that for every $\nnn$,
$$\frac{g(y+t_{n}w)-g(y)}{t_{n}}\leq -\frac{f(y+t_{n}w)-f(y)}{t_{n}}-\frac{\langle\nabla f(y+t_{n}w)-\nabla f(y),x_{t_{n}}-y\rangle}{t_{n}}
+\langle\nabla f(y+t_{n}w),w\rangle.$$
Taking limits, we deduce that
\begin{align}
\deriv g(y)(w)& \leq -\langle \nabla f(y),w\rangle -\langle \Hess
f(y)w,x-y\rangle+\langle\nabla f(y),w\rangle\notag\\
& =\langle \Hess f(y)(y-x),w\rangle\leq
\max\langle \Hess f(y)(y-\bfproj{C}(y)),w\rangle.\label{hamilton2}
\end{align}
Combining
\eqref{hamilton1} and \eqref{hamilton2}, we obtain
$$(\forall  w\in\RR^J)\quad
\deriv g(y)(w)= \max\langle \Hess
f(y)(y-\bfproj{C}(y)),w\rangle,$$
from which
$$\hat{\partial} g(y)=\Hess f(y)(y-\conv\bfproj{C}(y)).$$
Since $\bfproj{C}\colon U\To C$ is upper semicontinuous and
compact-valued by Proposition~\ref{farthest}(ii), we see that
$\conv\bfproj{C}:U \To \conv C$ is also upper semicontinuous
(see, e.g.,  \cite[Lemma~7.12]{Phelps}). Invoking now the
continuity of $\Hess f$, it follows that
$\partial_{L} g(y)=\varlimsup_{z\rightarrow y}\hat{\partial} g(z)=\Hess
f(y)[y-\conv\bfproj{C}(y)]$.
Proposition~\ref{distance} shows that $g$ is locally Lipschitz on $U$.
Using \cite[Theorem~8.49]{Rock98}, we deduce that
$$\overline{\partial} g(y)=\conv \partial_L g(y)=\partial_L g(y)=\Hess
f(y)[y-\conv\bfproj{C}(y)]$$
and
$$(\forall w\in\RR^J)\quad \rockderiv g(y)(w)=\max\langle \Hess
f(y)[y-\conv\bfproj{C}(y)],w\rangle,$$
which completes the proof.
\end{proof}

\begin{corollary}\label{spain}
Suppose that $f$ is twice continuously differentiable on $U$ and that
for every $y\in U$, $\Hess f(y)$ is positive definite.
Let $y\in U$. Then the following hold.
\begin{enumerate}
\item The function $\bfD{C}$ is differentiable at $y\in U$ if and only if $\bfproj{C}(y)$ is a singleton.
\item
The set $\menge{y\in U}{\bfproj{C}(y) \text{ is a singleton}}$
is residual in $U$, and it has full Lebesgue measure.
\end{enumerate}
\end{corollary}
\begin{proof}
(i): Assume first that $\bfD{C}$ is differentiable at $y\in U$. Then
$\hat{\partial} \bfD{C}(y)=\{\nabla \bfD{C}(y)\}$, and
Proposition~\ref{howtofind} yields
$$\nabla \bfD{C}(y)=\Hess f(y)[y-\conv \bfproj{C}(y)].$$
Since $\Hess f(y)$ is invertible,
$$\conv \bfproj{C}(y)=y-\Hess f(y)^{-1}\nabla \bfD{C}(y);$$
thus, $\bfproj{C}(y)$ must be a singleton.
Conversely, assume that $\bfproj{C}(y)$ is a singleton. Apply
Proposition~\ref{howtofind} to deduce that the limiting
subdifferential $\partial_L \bfD{C}(y)$ is a singleton. This implies that
$\bfD{C}$ is strictly differentiable at $y$ (see \cite[Theorem~9.18(b)]{Rock98})
and hence differentiable at $y$.

(ii): Since $\bfD{C}$ is locally Lipschitz on $U$
(see Proposition~\ref{distance}),
Rademacher's Theorem
(see \cite[Theorem~9.1.2]{lewis} or \cite[Corollary~3.4.19]{Yuri})
guarantees that $\bfD{C}$
is differentiable almost everywhere on $U$.
Moreover, since $\bfD{C}$ is Clarke regular
on $U$ (see Proposition~\ref{howtofind}),
we use \cite[Theorem~10]{loewen} to deduce that
$\bfD{C}$ is generically differentiable on $U$.
The result now follows from (i).
\end{proof}

\section{Characterizations}\label{completechris}

In this section, we give complete characterizations of sets with
unique farthest-point properties.
To do so, we need the following two key results on
expressing the convex-analytical subdifferential of the function
$-f^{\vee}+\iota_{-C}$ (see also \eqref{starting})
and of the conjugate $(-f^{\vee}+\iota_{-C})^*$ in terms of
$\bfproj{C}\circ \nabla f^*$.
These results
extend Hiriart-Urruty's
\cite[Proposition~4.4 and Corollary~4.5]{urruty2} to the framework
of Bregman distances.

\begin{lemma}\label{-f+}
Let $x\in -C$.  Then
$\partial (-f^{\vee}+\iota_{-C})(x) =
(\bfproj{C}\circ \nabla f^*)^{-1}(-x)$.
\end{lemma}
\begin{proof}
Let $s\in\RR^J$.
By \cite[Theorem~X.1.4.1]{urruty1},
$s\in \partial(-f^{\vee}+\iota_{-C})(x)$ if and only if
\begin{equation}\label{thursdaynow}
-f(-x)+\iota_{-C}(x)+(-f^{\vee}+\iota_{-C})^*(s)=\langle s,x\rangle.
\end{equation}
In view of \eqref{newconjugate}, equation~\eqref{thursdaynow}
is equivalent to
$-f(-x)+(\bfD{C}\circ\nabla f^*)(s)-f^*(s)=\langle x,s\rangle$,
and hence to
$\bfD{C}(\nabla f^*(s))= f(-x)+f^*(s)+\langle x,s\rangle
=D(-x,\nabla f^*(s))$, i.e., to
$-x\in\bfproj{C}(\nabla f^*(s))$.
\end{proof}

\begin{lemma}\label{tommy}
We have
$\partial (-f^{\vee}+\iota_{-C})^*=-\conv (\bfproj{C} \circ\nabla f^*)$.
\end{lemma}
\begin{proof}
Let $x$ and $s$ be in $\RR^J$.
By \cite[Lemma~X.1.5.3]{urruty1} or \cite[Corollary~3.47]{Rock98},
\begin{equation*}\label{calculus}
\cconv(-f^{\vee}+\iota_{-C}) =\conv (-f^{\vee}+\iota_{-C}).
\end{equation*}
On the other hand,
$$x\in \partial (-f^{\vee}+\iota_{-C})^{*}(s) \quad \Leftrightarrow \quad
s\in \partial (-f^{\vee}+\iota_{-C})^{**}(x)=\partial\, \cconv (-f^{\vee}+\iota_{-C})(x).$$
Altogether,
\begin{equation}\label{e:dienstag:a-}
x\in \partial (-f^{\vee}+\iota_{-C})^{*}(s)
\quad \Leftrightarrow \quad s \in \partial \conv (-f^{\vee}+\iota_{-C})(x).
\end{equation}
Now by \cite[Theorem~X.1.5.6]{urruty1},
$s \in \partial \conv (-f^{\vee}+\iota_{-C})(x)$ if and only if there
there exists nonnegative real numbers
$\lambda_{1},\ldots,\lambda_{J+1}$ and points
$x_1,\ldots,x_{J+1}$ in $\RR^J$
such that
\begin{equation*}
\sum_{j=1}^{J+1}\lambda_{j}=1,
x=\sum_{j=1}^{J+1}\lambda_{j}x_{j}\quad\text{and}\quad
s\in \bigcap_{j\colon \lambda_j>0} \partial (-f^{\vee}+\iota_{-C})(x_{j});
\end{equation*}
furthermore, Lemma~\ref{-f+} shows that
$s\in \partial (-f^{\vee}+\iota_{-C})(x_{j})$ $\Leftrightarrow$
$x_{j}\in -(\bfproj{C}\circ \nabla f^*)(s)$.
Therefore, the two conditions of \eqref{e:dienstag:a-} are also
equivalent to
$x\in -\sum_{j=1}^{J+1}\lambda_{j}(\bfproj{C}\circ \nabla f^*)(s)$.
\end{proof}

\begin{remark}
When $f=\thalb\|\cdot\|^2$ is the energy
(see Example~\ref{ex:dienstag}\ref{ex:dienstag:energy}), then
\eqref{newconjugate} turns into
$$(-f^{\vee}+\iota_{-C})^*=\thalb\Delta_{C}^2-\thalb\|\cdot\|^2,$$
where $\Delta_{C}\colon x\mapsto \sup\|x-C\|$.
In this case, the conclusion of Lemma~\ref{tommy} is classic;
see \cite[pages~262--264]{urruty3}
and \cite[Theorem~4.3]{urruty2}.
\end{remark}

We need the following result from \cite{solo}
(see also \cite[Section~3.9]{Zali}).

\begin{fact} [Soloviov]\label{vlad} Let $g:\RR^J\rightarrow\RX$ be lower semicontinuous,
and $g^*$ be essentially smooth. Then
$g$ is convex.
\end{fact}

We are now ready for the main result of this section.
\begin{theorem}[Characterizations of \bDK\ sets] \label{great}
The following are equivalent.
\begin{enumerate}
\item \label{freitag:1} $C$ is \bDK, i.e.,
$\bfproj{C}$ is a single-valued on $U$.
\item \label{freitag:2} $\bfproj{C}$ is single-valued and continuous on $U$.
\item \label{freitag:3} $\bfD{C}\circ\nabla f^*$ is continuously differentiable on $\RR^J$.
\item \label{freitag:4} $-f^{\vee}+\iota_{-C}$ is convex.
\item \label{freitag:5}$C$ is a singleton.
\suspend{enumerate}
If {\rm \ref{freitag:1}--\ref{freitag:5}} hold, then
\begin{equation}\label{west1}
\nabla (\bfD{C}\circ \nabla f^*)=\nabla f^*-\bfproj{C}\circ\nabla f^*.
\end{equation}
If $f$ is twice continuously differentiable and
the Hessian $\Hess f(y)$ is positive definite for every $y\in U$,
then {\rm \ref{freitag:1}--\ref{freitag:5}} are also equivalent to
\resume{enumerate}
\item \label{freitag:6} $\bfD{C}$ is differentiable on $U$,
\end{enumerate}
in which case $\bfD{C}$ is actually continuously differentiable on $U$ with
\begin{equation}\label{west2}
(\forall y\in U)\quad \nabla
\bfD{C}(y)=\Hess f(y)\big(y-\bfproj{C}(y)\big).
\end{equation}
\end{theorem}
\begin{proof}
``\ref{freitag:1}$\Rightarrow$\ref{freitag:2}'':
Apply Proposition~\ref{farthest}(ii).
``\ref{freitag:2}$\Rightarrow$\ref{freitag:3}'':
On the one hand, \eqref{newconjugate} implies
\begin{equation}\label{west3}
\bfD{C}\circ\nabla f^*=(-f^{\vee}+\iota_{-C})^*+f^*.
\end{equation}
On the other hand, Lemma~\ref{tommy} yields
\begin{equation} \label{west3a}
\nabla (-f^{\vee}+\iota_{-C})^*=-\bfproj{C} \circ\nabla f^*.
\end{equation}
Combining \eqref{west3} and \eqref{west3a}, we obtain
altogether \ref{freitag:3}, and also \eqref{west1}.
``\ref{freitag:3}$\Rightarrow$\ref{freitag:4}'':
This follows from \eqref{newconjugate} and Fact~\ref{vlad}.
``\ref{freitag:4}$\Rightarrow$\ref{freitag:5}'':
Assume to the contrary that $C$ is not a singleton,
fix two distinct points $y_0$ and $y_1$ in $C$, and
$t\in\RR$ with $0<t<1$.
Set $y_t := (1-t)y_0+ty_1$. Since
$-f^{\vee}+\iota_{-C}$ is a convex function,
its domain $-C$ is a convex set. Hence $y_t\in C$
and $-f(y_t)\leq -(1-t)f(y_0)-tf(y_1)$.
However, since $f$ is strictly convex, the last inequality is impossible.
Therefore, $C$ is a singleton.
``\ref{freitag:5}$\Rightarrow$\ref{freitag:1}'': This is obvious.

Finally, we assume that $f$ is twice differentiable on $U$ and that
the $\Hess f(y)$ is invertible, for every $y\in U$.
The equivalence of \ref{freitag:1} and \ref{freitag:6} follows from
Corollary~\ref{spain}(i), and \eqref{regularcase} yields the formula for
the gradient \eqref{west2}, which is continuous by \ref{freitag:2}.
\end{proof}

\begin{theorem}\label{christmas}
Set
\begin{equation}
\theta_{C}: \RR^J\to \RX\colon  x\mapsto \inf_{c\in C}\big(f(x+c)-f(c)\big).
\end{equation}
Then $\theta_C$ is proper, lower semicontinuous,
\begin{equation}\label{neededmost}
\theta_{C}=f\Box \big(-f^{\vee}+\iota_{-C}\big),
\end{equation}
where this infimal convolution is exact at every point in
$\dom \theta_{C}=\dom f-C$, and
\begin{equation}\label{isarel}
\theta_{C}^*= \bfD{C}\circ\nabla f^*.
\end{equation}
Moreover,
\begin{equation} \label{thetafunction}
\text{ $\theta_{C}$ is convex
$\;\;\Leftrightarrow\;\;$ $C$ is a singleton.}
\end{equation}
\end{theorem}
\begin{proof}
For every $x\in \RR^J$, we have
\begin{align*}
\big(f\Box
(-f^{\vee}+\iota_{-C})\big)(x)&=\inf_{y}\big(f(x-y)-f(-y)+\iota_{-C}(y)\big)\\
&=\inf_{-y\in C}\big(f(x-y)-f(-y)\big)\\
&=\inf_{c\in C}\big(f(x+c)-f(c)\big)\\
&=\theta_{C}(x),
\end{align*}
which verifies \eqref{neededmost} and the domain formula.
Since $\dom (-f^{\vee}+\iota_{-C})=-C$ is bounded,
\cite[Proposition~1.27]{Rock98} implies that
$f\Box (-f^{\vee}+\iota_{-C})$ is proper and lower semicontinuous, and
that the infimal convolution is exact at every point in its domain.
Using \eqref{newconjugate} and \cite[Theorem~11.23(a)]{Rock98}, we obtain
\begin{equation}
\bfD{C}\circ\nabla f^*=f^*+(-f^{\vee}+\iota_{-C})^* =
\big(f\Box (-f^{\vee}+\iota_{-C})\big)^*.
\end{equation}
This and
\eqref{neededmost} yield \eqref{isarel}.

It remains to prove \eqref{thetafunction}. The implication ``$\Leftarrow$''
is clear. We now tackle ``$\Rightarrow$''.
Since $U-C \subseteq \dom f - C =\dom \theta_C$ and since $C\subset U$,
we have $0\in\intdom\theta_C$. Take $x\in\dom\partial\theta_C$ and
$x^*\in\partial \theta_C(x)$. Then
\begin{equation}
(\forall y\in\RR^J)\quad
\scal{x^*}{y-x} \leq \theta_C(y)-\theta_C(x).
\end{equation}
On the other hand, there exists $\bar{c}\in C$ such that
$\theta_C(x) = f(x+\bar{c})-f(\bar{c})$ and also
$(\forall y\in \RR^J)$ $\theta_C(y)\leq f(y+\bar{c})-f(\bar{c})$.
Altogether,
\begin{equation}
(\forall y\in\RR^J)\quad
\scal{x^*}{y-x} \leq f(y+\bar{c})-f(x+\bar{c}),
\end{equation}
and this implies $x^*\in \big(\partial f(\cdot + \bar{c})\big)(x)$.
Since $f$ is essentially smooth, it follows that $\partial \theta_C(x)$ is
a singleton.
In view of \cite[Theorem~26.1]{Rock70},
$\theta_{C}$ is essentially smooth, and thus
differentiable on $\intdom \theta_{C}$.
Because $0\in\intdom \theta_C$,
$\theta_C$ is locally Lipschitz and differentiable at every point
in an open neighbourhood $V$ of $0$. Now set
\begin{equation}
g\colon \RR^J\to\left[-\infty,\pinf\right[
\colon x\mapsto \sup_{c\in C}\big( f(c)-f(c+x)\big).
\end{equation}
Then $\theta_C=-g$, $g$ is lower $C^1$
(see  \cite[Definition~10.29]{Rock98}), and
$$C=\menge{c\in C}{g(0)=0=f(c)-f(c)}.$$
By \cite[Theorem~{10.31}]{Rock98},
$\overline{\partial} g(0)=\conv\menge{-\nabla f(c)}{c\in C}=
-\conv \{\nabla f(C)\}$.
As $g$ is locally Lipschitz on $V$,
\cite[Theorem~2.3.1]{Frank} now yields
$$\overline{\partial}(-g)(0)=-\overline{\partial} g(0)=\conv\big\{\nabla
f(C)\big\}.$$
Using finally that $\theta_{C}=-g$ is convex,
and that $\partial =\overline{\partial}$ for convex functions,
\cite[Proposition~2.2.7]{Frank}, we obtain
$$\nabla\theta_{C}(0)=\partial \theta_{C}(0)=\partial (-g)(0)
=\overline{\partial}(-g)(0)=
\conv\big\{\nabla f(C)\big\},$$
i.e., $\conv\{\nabla f(C)\}=\nabla \theta_{C}(0).$
Therefore, $\nabla f(C)$ is a singleton, and so is $C$ by
Fact~\ref{isom}.
\end{proof}

\begin{remark}
If $f=\thalb\|\cdot\|^2$, then
\begin{align*}
\theta_{C}(x)&=\inf_{c\in C}\big(\thalb\|x+c\|^2-\thalb\|c\|^2\big)\\
&=\inf_{c\in C}\big(\thalb\|x\|^2+\scal{x}{c}\big)\\
&=\thalb\|x\|^2-\sup\scal{-C}{x}
\end{align*}
is the function introduced by Hiriart-Urruty
in \cite[Definition~4.1]{urruty2}.
Thus, the equivalence \eqref{thetafunction} extends
\cite[Proposition~4.2]{urruty2}.
\end{remark}

\section{Right Bregman Farthest-Point Maps}\label{finalright}

In this section, we relax our assumptions on $f$, i.e.,
we will only assume \textbf{A1} and \textbf{A3}.
It will be important to emphasis the dependence on $f$
for the Bregman distance and for the (left and right) Bregman
farthest-point map; consequently, we will write
$D_{f}$, $\bfproj{C}^{f}$, $\ffproj{C}^{f}$, and similarly for $f^*$.
While $D_f$ is generally not convex in its right (second) argument ---
which makes the theory asymmetric ---
it turns out that  $\ffproj{C}^{f}$ can be studied via
$\bfproj{\nabla f(C)}^{f^*}$.

\begin{proposition}\label{different}
Suppose that $f$ and $C$ satisfy
{\rm \textbf{A1}} and {\rm \textbf{A3}}.
Then
\begin{equation} \label{samstag}
\ffproj{C}^{f}=\nabla f^*\circ\bfproj{\nabla f(C)}^{f^*}\circ \nabla f
\quad\text{and}\quad
\bfproj{\nabla f(C)}^{f^*}=\nabla f\circ \ffproj{C}^{f}\circ\nabla f^*.
\end{equation}
\end{proposition}
\begin{proof}
Applying \cite[Theorem~3.7(v)]{Baus97} to $f^*$, we see that
\begin{equation*}
(\forall x^*\intdom f^*)(\forall y^*\in \intdom f^*)\quad
D_{f^*}(x^*,y^*)=D_{f}\big(\nabla f^*(y^*),\nabla f^*(x^*)\big).
\end{equation*}
Hence for every $y^*\in\intdom f^*$, we obtain
\begin{align*}
\bfproj{\nabla f(C)}^{f^*}(y^*) & ={\argmax_{x^*\in \nabla f(C)}
D_{f^*}(x^*,y^*)}\\
&= \argmax_{x^*\in \nabla f(C)} D_{f}\big(\nabla f^*(y^*),\nabla
f^*(x^*)\big)\\
& =\nabla f \Big(\ffproj{\nabla f^*\big(\nabla f(C)\big)}^{f}\big(\nabla f^*(y^*)\big)\Big)\\
&=\nabla f\Big(\ffproj{C}^{f}\big(\nabla f^*(y^*)\big)\Big),
\end{align*}
and this is the right identity in \eqref{samstag}; the left one
now follows Fact~\ref{isom}.
\end{proof}

\begin{theorem}\label{farpart}
Suppose that $f$ and $C$ satisfy {\rm \textbf{A1}} and {\rm \textbf{A3}}, that
$\dom f = \RR^J$, and that
$C$ is \fDK, i.e.,
for every $y\in\RR^J$,
$\ffproj{C}^{f}(y)$ is a singleton.
Then $C$ is a singleton.
\end{theorem}
\begin{proof}
Since $C$ is compact and
$\nabla f:\RR^J\to\intdom f^*$ is an isomorphism (see Fact~\ref{isom}),
we deduce that $\nabla f(C)$ is a compact subset of $\intdom f^*$.
Furthermore, by \eqref{samstag},
$\nabla f(C)$ is \bDK\ with respect to $f^*$.
Since $f^*$ satisfies \textbf{A1}--\textbf{A3}, we apply
Theorem~\ref{great} and conclude that $\nabla f(C)$ is a singleton.
Finally, again using Fact~\ref{isom},
we see that $C$ is a singleton.
\end{proof}

\begin{remark}
We do not know whether
Theorem~\ref{farpart} is true if the full-domain assumption on $f$ is
dropped.
\end{remark}

\section*{Acknowledgments}
Heinz Bauschke was partially supported by the Natural Sciences and
Engineering Research Council of Canada and by the Canada Research Chair
Program.
Xianfu Wang was partially
supported by the Natural Sciences and Engineering Research Council
of Canada.
Jane Ye was partially
supported by the Natural Sciences and Engineering Research Council
of Canada.
Xiaoming Yuan was partially supported by the Pacific Institute for the
Mathematical Sciences, by the University of Victoria,
by the University of British Columbia Okanagan, and
by the National Science Foundation of China Grant~10701055.

\end{document}